\documentclass {siamltex}
\usepackage{graphicx}
\usepackage{epsfig}
\usepackage{latexsym}
\usepackage{array}
\usepackage{amssymb}
\usepackage{amsmath}
\usepackage{amsbsy}

\newcommand{\tnorm}[1]{\vert\hspace{-0.3mm}\Vert#1\Vert\hspace{-0.3mm}\vert}
\newcommand{\jump}[1]{[\![#1]\!]}

\newtheorem{remark}{\it Remark\/}
\usepackage{eepic}
\title{Error estimates for shock capturing finite element
  approximations of the one dimensional Burgers' equation}

\author{Erik Burman\thanks{Department of Mathematics, 
University College London,
UK-- 
United Kingdom; ({\tt E.Burman@ucl.ac.uk})}
}
 
\begin{document}

\maketitle

\begin{abstract}
We propose an error analysis in weak norms of a shock capturing finite
element method for the Burgers' equation. The estimates can be related
to estimates of certain filtered quantities and are robust in the
inviscid limit. Using a total variation apriori bound on the discrete
solution and an interpolation inequality error estimates in 
$L^p$-norms can are obtained using interpolation. 
\end{abstract}


\newcommand{\cut}{c}
\def\IR{\mathbb R}
\def\Ext{\mbox{\textsf{E}}}

\section{Introduction}
There exists a vast litterature on
the design and convergence of numerical methods for nonlinear scalar
conservation laws dating back to the seminal work of Krushkov
\cite{Kru66}. Error estimates are obtained using entropy stability and the
so-called variable doubling technique. Asymptotic results have also
been obtained using entropy stability and compensated compactness. 

For work on convergence and error estimates using finite difference methods
we refer to \cite{Kru66, Kuz76, CM80, BO81, NT92, CoLeF93,EK00, LeF02},
using finite volume methods to \cite{CoCoLeF94, CoCoLeF95, CH99} and
finally for work on finite element methods see \cite{JS87, JS95,CG96a,Ohl01}.

For an introduction to these techniques we refer to the review article
by Cockburn \cite{Cock03} or the one by Tadmor \cite{Tad98}.

In this work we adopt a strategy that is reminiscent of the negative
norm estimates introduced by Tadmor and Nessyahu in \cite{NT92}. The key
argument of their analysis is to use a duality argument to get continuous
dependence on initial data for the adjoint perturbation equation of
the Burgers' equation in the $Lip'$-norm, i.e. the norm associated to the dual of
the space of Lipschitz-continuous functions. Provided the numerical
scheme has certain stability properties, in particular that the
discrete solution satisfies the discrete maximum principle and the
Oleinik E-condition (see \cite{Ole63}, this continuous dependence estimate
leads to estimates in a weak norm. Estimates in general $L^p$-norms
may then be recovered using interpolation.

Here we combine these ideas with the theory of dual weighted a
posteriori error estimates for finite element methods \cite{JS95,
  HMSW99} for the viscous Burgers' equation in one space dimension.
For a shock-capturing finite element method propose to estimate the error of certain filtered quantities
associated to weighted weak norms \cite{Bu98}.
 Indeed we apply the following differential filter to the error,
\begin{equation}\label{filter_def}
-\delta^2 \partial_{xx} \tilde u + \tilde u = u(\cdot,T) \quad \mbox{on } I
\end{equation}
with periodic boundary conditions on $\tilde u$ and $\partial_x \tilde
u$. The coefficient $\delta$ is proportional to the filter width.
Equation \eqref{filter_def} naturally leads to the error norm
\[
\tnorm{\tilde u - \tilde u_h}_\delta := \left(\|\delta \partial_x
  (\tilde u - \tilde u_h)\|^2 +  \|\tilde u - \tilde u_h\|^2\right)^{\frac12}
\]
where $\|\cdot\|$ denotes the $L^2$-norm. We prove an a posteriori
error estimate for this norm, proving boundedness of the stability
factors. The errors in the filtered quantities can be associated to
the estimation of certain weighted averages using a weight function in $H^1$.

For the discretization we will consider two different stabilized finite element
methods. In both cases the standard Galerkin method (with diagonal
mass for the time derivative) is supplemented with a term of
artificial viscosity type. First the classical linear artifical viscosity
resulting in a first order scheme, related to classical upwind
schemes for finite elements and vertex cell finite
volume methods and then a weakly consistent nonlinear viscosity in
the spirit of that proposed in \cite{Bu07b}. The latter is a shockcapturing
technique related to those proposed in \cite{JS87} or the entropy
viscosity of \cite{GPP11}. A key observation of this work is that the
dual stability required for the error estimates leads to design
criteria that have to be satisfied by the shock-capturing term.

Our main result then follows by using the
discrete stability of the numerical scheme to upper bound the
residuals of the a posteriori estimate resulting in the
following error estimate
\begin{equation}\label{filter_estimate}
\|\delta \partial_x (\tilde u - \tilde u_h)(T)\|_{L^2(I)}+\|(\tilde u - \tilde
u_h)(T)\|_{L^2(I)} \leq \tilde C(u_0,T) \exp({D_0 T}) \left(\frac{h}{\delta^2}\right)^\frac12
\end{equation}
where $\tilde u$ and $\tilde u_h$ are the filtered exact and
computational solution respectively. The constant in \eqref{filter_estimate} depend
only on the intial data, the mesh geometry and the final time. We will
use the notation $U_0 := \sup_{x \in I} |
\pi_h u_0(x)|$ and $D_0 := \sup_{x
  \in I} \tfrac12 \partial_x (u_0 + \pi_h u_0(x))$, where $\pi_h$ denotes the
$L^2$-projection onto the finite element space. We will choose $u_0$ as
a smooth function and by the stability of the $L^2$-projection on
regular meshes we have $U_0 \lesssim \sup_{x \in I} |u_0(x)|$ and $D_0 \lesssim \sup_{x
  \in I} |\partial_x  u_0(x)|$, so that estimates depending on $U_0$
and $D_0$ are indeed mesh
independent. Here and in the following we use the notation $a \lesssim
b$ 
defined by $a \leq Cb$ with $C$ a constant
independent of $h$, the physical parameters (except if they can be
assumed to make an $O(1)$ contribution) and of the exact
solution. We will also use $a \sim b$ for $a \lesssim b$ and $b
\lesssim a$. For simplicity we assume $u_0 \in C^\infty(I)$, with all
the derivatives matcing across the periodic boundaries,
this does not exclude the formation of sharp layers with gradients of
order $\nu^{-1}$ at later times.

The derivation of the estimate \eqref{filter_estimate} uses:
\begin{itemize}
\item[--] stability estimates for the finite element
  method,
\item[--] maximum principles for the finite element solution and its
  first derivative,
\item[--] a priori stability estimates on a linearized dual problem
  with regularized data,
\item[--] Galerkin orthogonality and approximability.
\end{itemize}

Using a Galiardo-Nirenberg interpolation estimate and the previous stability and error
estimates we also obtain the following error estimate in the
$L^p$-norm
\begin{equation}\label{L2error_estimate}
\|(u- u_h)(\cdot,t)\|_{L^p(I)} \leq C h^{\frac{1}{3p}}, \, \forall t>0.
\end{equation}
The constant depends on the all the constants of the previous
estimates, but is independent of the viscosity. 

All these results are obtained for the semi-discretization in space
only. The extension to the fully discrete case is
straightforward in the case of linear artificial viscosity using
previous results on finite difference methods, but not so
immediate when nonlinear viscosity is used. Indeed in the latter case
implicit schemes require regularization of the nonlinearity and the
effect of which must be assessed and for explicit schemes,
even $L^2$-stability for nonlinear viscosity methods is relatively
recent \cite{BGP13}. We therefore leave this aspect for future work.

\section{The Burgers' equation with dissipation}
Consider the simple model case
of the Burgers' equation with periodic boundary conditions, on the space-time domain $Q:= I \times
(0,T)$, with 
$I:=(0,1)$
\begin{equation}\label{burger}
\begin{array}{rcl}
\partial_t u + \frac12 \partial_x u^2 - \nu \partial_{xx} u &=& 0 \mbox{
  in } Q\\[3mm]
u(0,t) &=& u(1,t) \mbox{
  for } t \in (0,T)\\[3mm]
\partial_x u(0,t) &=& \partial_x u(1,t) \mbox{
  for } t \in (0,T)\\[3mm]
u(x,0) &=& u_0(x) \mbox{
  for } x \in I.
\end{array}
\end{equation}

The wellposedness of the equation \eqref{burger} for $\nu\ge0$ is well known it is also known that for $\nu>0$ by parabolic
regularization the solution is $C^\infty(I)$. This high regularity
however does not necessarily help us when approximating the solution,
since we are interested in computations using a mesh-size that is much
larger than the viscosity and still want the bounds to be independent
of high order Sobolev norms of the exact solutions and of $\nu$. Let
us first show how standard $L^2$-energy arguments fail when sharp
gradients develop in the solution.
\subsection{$L^2$-stability of Burgers' equation}
Consider a general perturbation $\eta(x)$ of the initial data of \eqref{burger}.
\begin{equation}\label{burger_pert}
\begin{array}{rcl}
\partial_t \hat u + \frac12 \partial_x \hat u^2 -
\nu \partial_{xx} \hat u &=& 0 \mbox{
  in } Q\\[3mm]
\hat u(0,t) &=& \hat u(1,t) \mbox{
  for } t \in (0,T)\\[3mm]
\partial_x \hat u(0,t) &=& \partial_x \hat u(1,t) \mbox{
  for } t \in (0,T)\\[3mm]
\hat u(x,0) &=& u_0(x) + \eta(x) \mbox{
  for } x \in I.
\end{array}
\end{equation}
Taking the difference of \eqref{burger_pert} and \eqref{burger} leads
to the perturbation equation for $\hat e := \hat u - u$
with $a(u,\hat u) := \tfrac12 (u +
\hat u)$,
\begin{equation}\label{pert_equation}
\begin{array}{rcl}
\partial_t \hat e + \partial_x (a(u,\hat u) \hat e) -
\nu \partial_{xx} \hat e &=& 0 \mbox{
  in } Q, \\[3mm]
\hat e(0,t) &=& \hat e(1,t) \mbox{
  for } t \in (0,T)\\[3mm]
\partial_x \hat e(0,t) &=& \partial_x \hat e(1,t) \mbox{
  for } t \in (0,T)\\[3mm]
\hat e(x,0) &=&  \eta(x) \mbox{
  for } x \in I.
\end{array}
\end{equation}
Multiplying equation \eqref{pert_equation} by $\hat e$ and
integrating over $Q$ leads to the energy equality
\[
\frac12 \|\hat e(T)\|_{L^2(I)}^2 + \|\nu^{\frac12} \partial_x
\hat e\|^2_{L^2(Q)} =\frac12 \|\eta\|_{L^2(I)}^2
-\frac12 \int_{Q} (\partial_x a(u,\hat u)) \hat e^2.
\]
We know that due to shock formation $\|\partial_x a(u,\hat u)\|_{L^\infty(I)}
\sim \nu^{-1}$ \cite{Ole63}. Any attempt to obtain control of $\|\hat
e(T)\|_{L^2(I)}^2$ in terms of the initial data will 
rely on Gronwall's lemma, leading to
\[
\|\hat e(T)\|_{L^2(I)}^2 \leq C_a \|\eta\|_{L^2(I)}^2
\]
with the exponential factor
\[
C_a :=\exp({\|\partial_x a(u,\hat u)\|_{L^\infty(Q)} T}) \sim  \exp(T/\nu).
\]
This factor obviously makes the estimate meaningless for large
gradients/small viscosities.
It tells us that the energy method only gives us useful information on
the stability up to the formation of shocks.  Using this type of argument in the analysis of
the finite element method leads to error estimates of the type derived
in \cite{BF07}, useful only for solutions with moderate gradients.
\subsection{Maximum principles for Burgers' equation}
It is well known that the equation \eqref{burger} satisfies a maximum
principle \cite{Ole63} on the form:
\begin{equation}\label{CMP}
\sup_{(x,t) \in  Q} |u(x,t)| \leq \sup_{x \in I} |u_0(x)|.
\end{equation}
This follows using standard techniques recalling the smoothness of the solution $u$ (or in the hyperbolic case, using the
method of characteristics). For our purposes we also need some precise information on the
derivative. Since the solution of \eqref{burger} is smooth we may
derive the equation in space to obtain the following equation for the
space derivative $w:=\partial_x u$:
\begin{equation}\label{burger_derivative}
\begin{array}{rcl}
\partial_t w + u \partial_x w -
\nu \partial_{xx} w &=& -w^2 \mbox{
  in } Q\\[3mm]
w(0,t) &=& w(1,t) \mbox{
  for } t \in (0,T)\\[3mm]
\partial_x w(0,t) &=& \partial_x w(1,t) \mbox{
  for } t \in (0,T)\\[3mm]
w(x,0) &=& \partial_x u_0(x) \mbox{
  for } x \in I.
\end{array}
\end{equation}
Assuming that for some time $t>0$ $w$ takes its maximum in some point $x \in I$ and noting
that $\partial_x w(x,t) =0$ and $\partial_{xx} w(x,t) < 0$ it follows that
$\partial_t w(x,t)<0$ at the maximum and we deduce the bound:
\begin{equation}\label{grad_bound}
\max_{(x,t) \in Q}\partial_x u \leq \max_{x \in I}\partial_x u_0.
\end{equation}
It follows by the smoothness of the initial data that the space
derivative is bounded above for all times.
\section{Artificial viscosity finite element method}
Discretize the interval $I$ with $N$ elements and let the local
mesh-size be defined by $h := 1/N$.
We denote the computational nodes by $x_i := i \, h$, $i=0,\hdots,N$,
defining the elements $I_j:=[x_j,x_{j+1}]$, $j=0,\hdots,N-1$,
and the standard nodal basis functions $\{v_i\}_{i=0}^{N}$, such that $v_i(x_j)
= \delta_{ij}$, with $\delta_{ij}$ the Kronecker delta. To impose
periodic boundary conditions we identify the node $x_0$ with $x_N$
and define the corresponding basis function $v_{0N}:(x_0,x_1) \cup
(x_{N-1},x_N) \mapsto \mathbb{R}$ by $v_0$ on $(x_0,x_1)$
and by $v_N$ on $(x_{N-1},x_N)$. This basis function then replaces $v_0$ and
$v_N$, leading to a total of $N$ degrees of freedom. For simplicity we
use the notation $v_0$ for the basis function $v_{0N}$. The finite
element space is given by
\[
V_h:=\left\{ \sum_{i=0}^{N-1} u_i v_i, \mbox{ where } \{u_i\}_{i=0}^{N-1} \in \mathbb{R}^N\right\}.
\]
We define the standard $L^2$ inner product on $X \subset I$ by
\[
(v_h,w_h)_X := \int_X v_h w_h\,\mbox{d}x.
\]
The
discrete form corresponding to mass-lumping reads
\[
(v_h,w_h)_h := \sum_{i=0}^{N-1} v_h(x_i) w_h(x_i) h.
\]
The associated norms are defined by $\|v\|_X:=(v,v)_X^{\frac12}$, for all
$v \in L^2(X)$, if $X$ coincides with $I$ the subscript is dropped, and $\|v_h\|_h:= (v_h,v_h)_h^{\frac12}$ for all $v_h
\in V_h$.
Note that, by norm equivalence on discrete spaces, for all $ v_h \in V_h$
there holds
\[
\|v_h\|_h \lesssim \|v_h\| \lesssim \|v_h\|_h.
\]
Using the above notation the artificial viscosity finite element space
semi-discretization of \eqref{burger} reads, given $u_{0} \in C^\infty(I)$
find $u_h(t) \in V_h$ such that $(u_h(0),v_h)_I = (u_{0},v_h)_I$ and
\begin{equation}\label{art_visc_FEM}
(\partial_t u_h,v_h)_h + \left(\partial_x \frac{u_h^2}{2},v_h\right)_I + (\hat
\nu \partial_x u_h, \partial_x v_h)_I = 0,\mbox{ for all } v_h \in
V_h \mbox{ and } t>0,
\end{equation}
where we propose two different forms of $\hat \nu$:
\begin{enumerate}
\item linear artificial viscosity: 
\begin{equation}\label{lin_visc}
\hat \nu := \max( U_{0} h/2,
  \nu);
\end{equation}
\item nonlinear artificial viscosity: \\
Let $0\leq\epsilon$ and
\begin{equation}\label{nu0}
\nu_0(u_h)\vert_{I_i}:= \frac{1}{2} \|u_h\|_{L^\infty(I_i)} \max_{x \in \{x_i,x_{i+1}\}}
\frac{|\jump{\partial_x u_h}\vert_{x}|}{2 \{|\partial_x u_h| \}\vert_{x}+\epsilon},
\end{equation}
where $\jump{\partial_x u_h}\vert_{x_i}$ denotes the jump of $\partial_x u_h$ over the
node $x_i$ and  $\{|\partial_x u_h |\}\vert_{x_i}$ denotes the average
of $|\partial_x u_h |$ over $x_i$. If $\epsilon=0$ and $\{|\partial_x
u_h |\}\vert_{x_i}=0$
we replace the quotient $|[\partial_x u_h]\vert_{x_i}|/\{|\partial_x
u_h |\}\vert_{x_i}$ by zero.

Further let
\[
\xi(u_h)\vert_{I_i}:=\left\{
\begin{array}{ll}
1\quad\mbox{if $\partial_x u_h\vert_{I_i}>0$, $\partial_x
  u_h\vert_{I_i} > \partial_x u_h\vert_{I_{i+1}}>0$} \\ \quad \mbox{ and $\partial_x
  u_h\vert_{I_i} \ge \partial_x u_h\vert_{I_{i-1}}>0$}\\ 
0\quad \mbox{otherwise}
\end{array}
 \right.
\]
\begin{equation}\label{nu1}
\nu_1(u_h)\vert_{I_i}:= \xi(u_h)\vert_{I_i}
\frac12 (\nu_0\vert_{I_{i-1}}\frac{\partial_x
  u_h\vert_{I_{i-1}}}{\partial_x u_h\vert_{I_{i}}} +\nu_0\vert_{I_{i+1}}\frac{\partial_x u_h\vert_{I_{i+1}}}{\partial_x u_h\vert_{I_{i}}}\Bigr).
\end{equation}
Finally define:
\begin{equation}\label{nonline_visc}
\hat \nu(u_h)\vert_{I_i}:=\max(\nu, h(\nu_0\vert_{I_i}+
\nu_1\vert_{I_i})).    
\end{equation}
\end{enumerate}
The rationale for the nonlinear viscosity is to add first order
viscosity at local extrema of the solution $u_h$ so that \eqref{CMP}
holds also for the discrete solution and enough viscosity
at positive extrema of $\partial_x u_h$, making \eqref{grad_bound}
carry over to the discrete setting. The most important term is $\nu_0$,
ensuring the discrete maximum principle. The other part $\nu_1$ is
merely a correction that ensures that the viscosity at local maxima of
$\partial_x u_h$ dominates that of the surrounding elements.
 Indeed the role of the
function $\xi(u_h)$ is to act as an indicator function for the
elements where the local maxima of $\partial_x u_h$ are taken and
modify the viscosity there. By construction, if $\xi(u_h) = 1$ in one element it must
be zero in the neighbouring elements. Note that formally $\hat \nu(u_h) \approx
O(\max(U_0 h^{3/2}), \nu)$ in the smooth part of the solution, so that
in principle we can expect higher order convergence away from local
extrema. Note that the perturbation $\nu_1$ close to extrema of the
gradient of the solution has no impact on the formal order. 
The high order convergence properties and the effect of the
regularization parameter $\epsilon$ will be explored in the numerical
section.
\begin{remark}
Note that in the linear case the viscosity $\hat \nu$ may be written as 
\[
\hat \nu := \nu \max(1,Re_h), \quad \mbox{ with } Re_h := \frac{U_{0}
h}{2\nu}
\]
reflecting that in the high Reynolds number regime the viscosity is
increased artificially to be order $h$.
\end{remark}
\subsection{Existence of solution to the semidiscretized system}
For the linear method existence and uniqueness of solutions of
\eqref{art_visc_FEM}
follows using standard methods. First we observe that the nonlinear
function of the dynamical system is locally Lipschitz and then anticipating
the global upper bound \eqref{umax} we conclude that a global solution
exists and is unique.

The nonlinear method obtained when \eqref{art_visc_FEM} is used with
the viscosity \eqref{nonline_visc} results in a dynamical system with
discontinuous righthand side (even for $\epsilon>0$ the contribution
from $\nu_1$ introduces discontinuities). Existence of solutions
to \eqref{art_visc_FEM} with the nonlinear viscosity
\eqref{nonline_visc} is obtained using Filippov
theory \cite{Fi60}. Anticipating the results of the next section, we may
conclude that a solution exists, since by the discrete maximum
principle \eqref{umax}, for fixed $h$, there holds 
\[
\Bigl|\left(\partial_x \frac{u_h^2}{2},v_i\right)_I + (\hat
\nu \partial_x u_h, \partial_x v_i)_I\Bigr| \leq M \|v_i\|_{L^2(I)}
\]
and hence $|\partial_t u_h(x,t)|<M$ for all $(x,t)\in Q$. The question of uniqueness is
 more involved, but we
conjecture that the solution is forward unique since by construction
it satisfies the Oleinik E-condition. 
Indeed the decrease of the maximum derivative by construction, rules out the so called repulsive sliding mode that is known to
cause nonuniqueness of the solution. 
We will not explore these issues further here but
refer the interested reader to \cite{DL09}.
Typically in
practice the system \eqref{art_visc_FEM} will be discretized in time using an
explicit time stepping scheme which, by definition, will produce a
unique discrete solution. In the
following we will prove that any solution to \eqref{art_visc_FEM} will
satisfy certain uniform bounds and converge to the exact solution at a
certain rate.
%
%
\subsection{Maximum principles for the discrete solution}
Maximum principles give local estimates of the behavior of the
solution and they rarely carry over to the discrete method. There are 
however stabilized methods that are specially designed to make a
discrete maximum principle hold, see for instance \cite{XZ99,BE05} for linear
convection--diffusion problems and \cite{KT02,Bu07b} for maximum
principle satisfying finite element methods for conservation laws.

In \cite{Bu07b} it was shown that the discrete equivalent of \eqref{CMP},
together with energy stability of the discrete solution is sufficient to prove
the convergence of the approximation sequence to the entropy
solution. For our purposes herein however it is not sufficient, but we
also need to prove a discrete equivalent of the bound
\eqref{grad_bound} on the gradient. We collect the monotonicity
results we need in the following lemma. For clarity of the exposition we first give the proofs
in the case $\epsilon=0$ and then discuss how the regularization modifies
the bounds.
\begin{lemma}\label{DMP}
Let $u_h$ be the solution of \eqref{art_visc_FEM} either using the
linear viscosity \eqref{lin_visc}
or the nonlinear viscosity \eqref{nonline_visc} and $\epsilon=0$. Then the following
bounds hold:
\begin{equation}\label{umax}
\sup_{(x,t) \in Q} |u_h(x,t)| \leq U_0 \lesssim \max_{x \in I} |u_0(x)|,
\end{equation}
\begin{equation}\label{grad_max}
\sup_{(x,t) \in Q} \partial_x u_h(x,t) \leq \max_{x \in I} \partial_x u_h(x,0)
\lesssim \max_{x \in I}  |\partial_x u_0(x)|.
\end{equation}
\end{lemma}
\begin{proof}
The proof of \eqref{umax} is an immediate consequence of the fact that
the space discretization has the DMP-property introduced in
\cite{BE05}. 

For the case of linear artificial viscosity, first
assume that for some time $t^*$ there holds $\max_{x \in I} |u_h(x,t^*)| \leq
U_0$, then show that this implies 
\begin{equation}\label{maxtstar}
\max_{t\ge t*} \max_{x \in I} |u_h(x,t)| \leq
U_0
\end{equation}
 and conclude noting that the assumed inequality holds for $t^*=0$,
 since
$$\max_{x\in I} |u_h(x,0)| =:U_0.$$

First we compute
\begin{multline}\label{nonlin_int}
\int_I u_h \partial_x u_h v_i ~\mbox{d}x = \frac{h^2}{3} (\partial_x
u_h\vert_{I_{i-1}})^2 + \frac{h^2}{6} (\partial_x
u_h\vert_{I_{i}})^2 \\+ \frac{h}{2} u_h(x_{i-1},t^*) \partial_x
u_h\vert_{I_{i-1}} +\frac{h}{2} u_h(x_i,t^*) \partial_x
u_h\vert_{I_{i}}
\end{multline}
Assuming that
$u_h$ has a local max in $x_i$ at $t=t^*$ it follows that, for linear viscosity
\begin{multline*}
h \partial_t u_h(x_i,t^*) = - (u_h \partial_x u_h,v_i)_I - (\hat
\nu \partial_x u_h, \partial_x v_i)_I \\ \leq -\frac{h}{2}\left(U_{0} - \max(|u_h(x_{i-1},t^*)|,|u_h(x_i,t^*)|)\right)   (|\partial_x
u_h\vert_{I_{i-1}}|+|\partial_x u_h\vert_{I_{i}}|) \leq 0.
\end{multline*}
It follows that $\partial_t u_h(x_i,t^*) \leq 0$ and hence the local
maximum can not grow. The case of a local minimum is similar. 

For nonlinear viscosity on the form \eqref{nu0}, since at a local
maximum $|\jump{[\partial_x u_h]}| = 2 \{|\partial_x u_h\}$, we deduce that
$$\hat \nu(u_h)\vert_{I_{i}} = \|u_h(\cdot,t^*)\|_{L^\infty(I_i)} h/2 = \max(|u(x_i,t^*)|,|u(x_{i+1},t^*)|) h/2$$
and the same conclusion follows. To reduce the notation below we drop
the argument $t^*$.

We will show
\eqref{grad_max}
by first proving that the maximum gradient must be decreasing, and
then applying the stability of the
$L^2$ projection. Hence it is sufficient to prove that $\partial_t
\max_{i} \partial_x u_h\vert_{I_i} \leq 0$ in $Q$ to
conclude.

We first give the proof for the linear artificial viscosity.
We will prove that the discrete gradient is bounded by the gradient of
the discrete initial data. 
\[
\sup_{(x,t) \in Q} \partial_x u_h(x,t)  \leq \sup_{x \in
  I} \partial_x \pi_h u_0.
\]
Starting from \eqref{art_visc_FEM}, we let $I_i$ be any element where
$\partial_x u_h\vert_{I_i}$ has a local maximum, in the sense $\partial_x
u_h\vert_{I_i} \ge \partial_x u_h\vert_{I_{i\pm 1}} \ge 0$. 
\begin{multline*}
 \partial_t \partial_x u_h\vert_{I_{i}} = -\frac{1}{h} \int_{x_{i-1}}^{x_{i+2}}
u_h \partial_x u_h (v_{i+1} - v_i) ~ \mbox{d} x - \\\frac{1}{h} \int_{x_{i-1}}^{x_{i+2}}
\hat \nu \partial_x u_h \partial_x  (v_{i+1} - v_i) ~ \mbox{d} x = T_1+T_2.
\end{multline*}
Decomposing the integral $T_1$ on the contributions from $v_i$ and
$v_{i+1}$ we have using \eqref{nonlin_int} and some minor manipulations
\begin{multline*}
 T_1 = - \frac16 h (\partial_x
u_h\vert_{I_{i-1}})^2 - \frac23 h (\partial_x
u_h\vert_{I_{i}})^2 - \frac16 h (\partial_x
u_h\vert_{I_{i+1}})^2\\ - \frac12 u_h(x_i) (\partial_x
u_h\vert_{I_{i}} - \partial_x
u_h\vert_{I_{i-1}} )  - \frac12 u_h(x_{i+1}) (\partial_x
u_h\vert_{I_{i+1}} - \partial_x
u_h\vert_{I_{i}} ).
\end{multline*}
Since the derivative
takes its max value in $I_i$ we have for $T_2$
\begin{multline*}
T_2= - \frac{1}{h}\int_{x_{i-1}}^{x_{i+2}}
\hat \nu \partial_x u_h \partial_x  (v_{i+1} - v_i) ~ \mbox{d} x  \\=
h^{-1}
((\hat \nu(u_h) \partial_x u_h)\vert_{I_{i-1}}-2 (\hat \nu(u_h) \partial_x
u_h)\vert_{I_{i}}+ (\hat \nu(u_h) \partial_x u_h\vert_{I_{i+1}})) \leq 0.
\end{multline*}
Collecting the above expressions and using that $\hat \nu \ge \tfrac12 U_{0} h$, and  we obtain 
\begin{multline*}
 T_1+T_2 \leq -\frac16 h (\partial_x
 u_h\vert_{I_{i+1}})^2  -\frac23 h (\partial_x
 u_h\vert_{I_{i}} )^2- \frac{1}{6} h (\partial_x u_h\vert_{I_{i-1}})^2
 \\
 - \frac12 \underbrace{ (U_0 + u_h(x_i))}_{\ge 0}  \underbrace{(\partial_x
u_h\vert_{I_{i}} - \partial_x
u_h\vert_{I_{i-1}} )}_{\ge 0} \\
 - \frac12 \underbrace{ ( U_0 - u_h(x_{i+1}))}_{\ge 0} \underbrace{(\partial_x
u_h\vert_{I_{i}} - \partial_x
u_h\vert_{I_{i+1}} )}_{\ge 0}
\leq 0.
\end{multline*}
This proves that $\partial_t \max_{i} \partial_x u_h\vert_{I_i} \leq
0$ and therefore the maximum
space derivative is alway decreasing. The global upper bound \eqref{grad_max}
is immediate by the stability of the $L^2$-projection. Note that by
the non-monotonicity of the $L^2$-projection there is now an absolute
value on the derivative of the initial data.

In the case of the nonlinear viscosity given by \eqref{nonline_visc} we
first show that the time derivative of the gradient must be negative
in cells with $\xi(u_h)=1$. Then we show that any adjacent element,
cannot grow either due to the design of the nonlinear switch. 
In this case have after integration
\begin{multline}\label{nonlin}
 T_1+T_2 = -\frac16 h (\partial_x
 u_h\vert_{I_{i+1}})^2  -\frac23 h (\partial_x
 u_h\vert_{I_{i}} )^2- \frac{1}{6} h (\partial_x u_h\vert_{I_{i-1}})^2
 \\
+ h^{-1}
((\hat \nu(u_h) \partial_x u_h)\vert_{I_{i-1}}-2 (\hat \nu(u_h) \partial_x
u_h)\vert_{I_{i}}+ (\hat \nu(u_h) \partial_x u_h\vert_{I_{i+1}})) \\
- \frac12 u_h(x_i)(\partial_x
u_h\vert_{I_{i}} - \partial_x
u_h\vert_{I_{i-1}} ) + \frac12  u_h(x_{i+1})(\partial_x
u_h\vert_{I_{i}} - \partial_x
u_h\vert_{I_{i+1}} ).
\end{multline}
%
First observe that the case where either $x_{i}$ or $x_{i+1}$ is a
local extremum can be excluded, since then $\nu_0\vert_{I_i} =
\tfrac12 \|u_h\|_{L^{\infty}(I)}$ and 
by observing the sign of the
contribution of the derivative from the neighbouring cells in the
viscosity terms of line two and three of \eqref{nonlin}. The other
terms are controlled as in the linear theory with minor modifications.

Since the gradient has a local max in $I_i$, in the sense that $\xi(u_h)\vert_{I_i}=1$ there holds
$$
\hat\nu(u_h)\vert_{I_i} = \frac12 (\nu_0(u_h)\vert_{I_{i-1}}\frac{\partial_x u_h\vert_{I_{i-1}}}{\partial_x u_h\vert_{I_{i}}}+\nu_0(u_h)\vert_{I_{i+1}}\frac{\partial_x u_h\vert_{I_{i+1}}}{\partial_x u_h\vert_{I_{i}}})+\nu_0(u_h)\vert_{I_{i}}
$$  
and since by construction $\xi(u_h)\vert_{I_{i \pm 1 }}=0$, we have in
the neighbouring cells,
$$
\hat\nu(u_h)\vert_{I_{i \pm 1}} = \nu_0(u_h)\vert_{I_{i \pm 1}}.
$$
Using the values of $\hat \nu$ the contribution from the viscous part of the
differential operator may be bounded as
 \begin{multline}\label{bound_art_osc}
(\hat \nu(u_h) \partial_x u_h)\vert_{I_{i-1}}-2 (\hat \nu(u_h) \partial_x
u_h)\vert_{I_{i}}+ (\hat \nu(u_h) \partial_x u_h\vert_{I_{i+1}}) = - 2
h \nu_0(u_h)\vert_{I_i} \partial_x
 u_h\vert_{I_i}
.
\end{multline}
For the two last terms in the right hand side of \eqref{nonlin} we note
that, assuming first $u_h(x_{i+1}) \ge u_h(x_i)\ge 0$,
\begin{multline}\label{bound_phys_osc}
- \frac12 \underbrace{u_h(x_i)(\partial_x
u_h\vert_{I_{i}} - \partial_x
u_h\vert_{I_{i-1}} )}_{\ge 0}+ \frac12  u_h(x_{i+1})(\partial_x
u_h\vert_{I_{i}} - \partial_x
u_h\vert_{I_{i+1}} ) \\ 
\leq  \nu_0(u_h)\vert_{I_i}
(|\partial_x u_h\vert_{I_i}|+|\partial_x u_h\vert_{I_{i+1}}|)
\end{multline}
Collecting \eqref{bound_art_osc} and \eqref{bound_phys_osc},
\eqref{nonlin} can be upper bounded in the following fashion,
recalling that $\partial_x u_h\vert_{I_i} > \partial_x u_h\vert_{I_{i+1}}>0$:
\begin{multline}
 T_1+T_2 \leq -\frac16 h (\partial_x
 u_h\vert_{I_{i+1}})^2  -\frac23 h (\partial_x
 u_h\vert_{I_{i}} )^2- \frac{1}{6} h (\partial_x u_h\vert_{I_{i-1}})^2
 \\
- 2 \nu_0(u_h)\vert_{I_i} \partial_x
u_h\vert_{I_i}
+\nu_0(u_h)\vert_{I_i}
(|\partial_x u_h\vert_{I_i}|+|\partial_x u_h\vert_{I_{i+1}}|) \leq 0.
\end{multline}
The case $0\ge u_h(x_{i+1}) \ge u_h(x_i)$ is similar observing that in
that case the last term in the right hand side of \eqref{nonlin} is
negative. In case $u_h(x_i) < 0 < u_h(x_{i+1})$ we observe that only
the treatment of the first contribution of the last line of \eqref{nonlin} must be modified. We
note that
\[
- \frac12 u_h(x_i)(\partial_x
u_h\vert_{I_{i}} - \partial_x
u_h\vert_{I_{i-1}} ) \\
\leq \frac12 h (\partial_x
u_h\vert_{I_{i}})^2 
\]
and that this term is cancelled by the second term of the right hand
side in the first line of \eqref{nonlin}.

Finally we must check that the gradient can not increase in any
portion of the domain where the gradient is constant at the maximum
value over several
elements from $I_m$ to $I_n$, $m<n$, at some time $t^*$. First note that for all elements
$I_{m+2},...,I_{n-2}$ the derivative is decreasing, since only the
first three terms of the right hand side of \eqref{nonlin} are non-zero. 
By construction $\xi(u_h)\vert_{I_n}=1$ and hence the derivative is
decreasing in $I_n$. As a consequence the derivative in $I_{n-1}$ is
either decreasing at the time $t^*$ or will have
$\xi\vert_{I_{n-1}}=1$ at $t^*+\varepsilon$ for all $\varepsilon>0$
and hence be non-increasing.
Similarly for $I_{m}$ and $I_{m+1}$, the derivatives can not grow at
the same rate in both cells since then $\xi\vert_{I_{m+1}}=1$ at
$t^*+\varepsilon$ for all $\varepsilon>0$, since its right hand side
neighbour has decreasing space derivative. If, on the other hand the time derivative
of the derivative is the largest in $I_m$ at time $t=t^*$ then $\xi\vert_{I_{m}}=1$ at
$t^*+\varepsilon$ for all $\varepsilon>0$, and hence the derivative
can not grow. This case of several adjacent cells over which the
gradient is constant is what can give rise to the so called attractive
sliding mode in the Filippov theory.
\end{proof}
\begin{remark} First observe that if $\nu_1=0$ it is not difficult to
  find $u_h$ for which $\partial_t \max_i \partial_x u_h
  \vert_{I_i}>0$, so the above technique of proof requires the
  contribution from $\nu_1$, whether or not it is really necessary in
  practice remains unclear. 
\end{remark}

It was shown in \cite{Bu07b} that a consequence of the bound
\eqref{DMP} is that the total variation of $u_h$ diminuishes. We
recall the result without proof.
\begin{corollary}\label{TVbound}
Let $u_h$ be the solution of \eqref{art_visc_FEM}, then there holds,
for all $t\ge0$
\[
TV(u_h(\cdot,t)):= \int_I |\partial_x u_h(\cdot,t)| ~\mbox{d}x \leq TV(u_h(\cdot,0)).
\]
\end{corollary}
\subsubsection{The effect of non-zero regularization parameter
  $\epsilon$}
In practice it may be practical to use a value on $\epsilon$ that is
related to the mesh size, in particular if implicit solvers are used
it is known that the regularized shock-capturing term has smoother
convergence properties. This will result in a modification of the
upper bounds \eqref{umax} and \eqref{grad_max}, but as we show below,
the maximum principles can only be violated by an $\mathcal{O}(\epsilon)$. 
\begin{proposition}\label{reg_pert_bounds}
Let $0<\epsilon T < 1$  in \eqref{nu0}, let $u_h$ be the
solution of \eqref{art_visc_FEM}, then there holds
\begin{equation}\label{reg_bound}
\|u_h(\cdot,T)\|_{L^\infty(I)} \leq (1 + \epsilon T) U_0
\end{equation}
and
\begin{equation}\label{grad_reg_bound}
\max_{(x,t) \in Q} \partial_x u_h(x,t) \leq \max_{x \in I} \partial_x u_h(x,0)+  U_0 (1+\epsilon T)
\epsilon T.
\end{equation}
\end{proposition}
\begin{proof}
Assume for simplicity that $u_h \ge 0$
and that $u_h$ takes a global (positive) maximum in $x_i$ that will grow with the maximum
rate throughout the computation. Introduce the notation
$$g_1:=|\partial_x u_h\vert_{I_{i-1}}|,\quad
g_2:=|\partial_x u_h\vert_{I_{i}}|.$$ Recall that
at a local maximum of $u_h$, $\xi(u_h)=0$ and therefore $\hat \nu(u_h)
= \max(\nu,\nu_0(u_h))$. Assume that the maximum is taken for
$\nu_0(u_h)$. Then by \eqref{nonlin_int}
\begin{multline*}
\partial_t u_h(x_i) \leq \frac12 u_h(x_i) (g_1+g_2) -
h^{-1} \hat \nu_h(u_h)\vert_{I_{i-1}}  g_1 - h^{-1} \hat \nu_h(u_h)\vert_{I_{i}} g_2) \\
 \leq \frac12 u_h(x_i) (g_1+g_2 -
 \frac{(g_1+g_2)^2}{g_1+g_2+\epsilon}) \leq \frac12
u_h(x_i) \epsilon.
\end{multline*}
By Gronwall's lemma it follows that
\[
u_h(x_i,T) \leq U_0 e^{\frac12 \epsilon T}.
\]
Since $e^x < 1 + x/(1-x)$ we conclude, using the assumption that $\epsilon
T < 1$,
\[
u_h(x_i,T) \leq U_0(1+\epsilon T).
\]
To obtain the inequality \eqref{grad_reg_bound} we reason in a similar
fashion starting from the equation \eqref{nonlin} and using
\eqref{reg_bound}. The regularization only comes into effect at the
step \eqref{bound_phys_osc} and we observe that in this case
\begin{multline}\label{bound_phys_reg}
- \frac12 u_h(x_i)(\partial_x
u_h\vert_{I_{i}} - \partial_x
u_h\vert_{I_{i-1}} ) + \frac12  u_h(x_{i+1})(\partial_x
u_h\vert_{I_{i}} - \partial_x
u_h\vert_{I_{i+1}} ) \\ 
\leq \nu_0(u_h)\vert_{I_i}
(|\partial_x u_h \vert_{I_i}|+|\partial_x u_h \vert_{I_{i+1}}|+\epsilon).
\end{multline}
This then leads to the bound
\[
\partial_t \partial_x u_h\vert_{I_i} \leq -\frac16 h (\partial_x
 u_h\vert_{I_{i+1}})^2  -\frac23 h (\partial_x
 u_h\vert_{I_{i}} )^2- \frac{1}{6} h (\partial_x u_h\vert_{I_{i-1}})^2
 + \frac12 u_h(x_{i+1}) \epsilon.
\]
Integrating in time shows that
\[
\max_{(x,t) \in Q} \partial_x u_h \leq \max_{x \in I} \partial_x u_h(x,0) + U_0
(1+\epsilon T) \epsilon T.
\]
\end{proof}\\
A regularization of $\xi_h$ can also be performed and analysed with
similar outcome for the estimate \eqref{grad_max}.
These perturbations of the discrete maximum principle then modifies
the result \eqref{TVbound}. Assuming that the maximum violation takes
place in every node in the mesh it is straightforward to show that the
total variation remains upper bounded uniformly in $h$ and with linear
growth in $T$, provided
$\epsilon \leq O(h)$.
\subsection{Energy stability}
Our estimates rely on stability of the numerical scheme and regularity
of the dual perturbation equation. We need to control certain Sobolev
norms of the discrete solution in energy type estimates similar to
that of the continuous problem. The proof of the below estimates can
be simplified in the linear case and the inverse estimate on the
$L^\infty$-norm that is only valid in one dimension can then be
avoided. Here we only give the proof valid both in the linear and in
the nonlinear case.
\begin{lemma}\label{discrete_stab}
The solution $u_h$ of the formulation \eqref{art_visc_FEM} with either
the linear artificial viscosity given by \eqref{lin_visc} or the
nonlinear one of \eqref{nonline_visc} with $\epsilon=0$, satisfies the upper bounds
\begin{equation}\label{energy_est}
\|u_h(T)\| + \|\hat \nu^{\frac12} \partial_x
u_h\|_Q  \lesssim  \|u_0\|
\end{equation}
\begin{equation}\label{timeder_stab}
\|\partial_t u_h\|_Q\lesssim  (U_0 T^{\frac12} h^{-\frac12}+\nu^\frac12)\| \partial_x
u_0\|.
\end{equation}
\end{lemma}
\begin{proof}
The estimate \eqref{energy_est} is immediate by taking $v_h = u_h$
and noticing, by integration by parts and the periodic boundary
conditions, that the nonlinear transport term vanishes. By norm equivalence and
the stability of the $L^2$-projection $$\|u_h(T)\| \lesssim \|u_h(T)\|_h
\mbox{ and } \|u_h(0)\|_h \lesssim \|u_0\|.$$ 

The second estimate follows by taking $v_h = \partial_t u_h$ to obtain
\begin{equation}\label{time_stab1}
\int_0^T \|\partial_t u_h\|_h^2 ~\mbox{d}t =- \int_0^T (u_h\partial_x u_h, \partial_t
u_h)_I ~\mbox{d}t -   \int_0^T (\hat \nu \partial_x
u_h, \partial_x \partial_t u_h)_I ~\mbox{d}t.
\end{equation}
First note that by Corollary \ref{TVbound} and \eqref{umax} we have,
since $TV(u_h) \leq \mbox{meas}(I)^{\frac12} \|\partial_x u_h\|$,
\begin{multline}\label{time_transp_bound}
\int_0^T (u_h\partial_x u_h, \partial_t
u_h)_I ~\mbox{d}t \leq U_0 TV(u_h(\cdot,0)) \int_0^T \|\partial_t
u_h(\cdot,t)\|_{L^\infty(I)} ~\mbox{d}t \\
\lesssim U_0 \| \partial_x
u_h(\cdot,0)\| T^{\frac12} h^{-\frac12}  \|\partial_t
u_h\|_{Q}.
\end{multline}
For the last term in the right hand side of \eqref{time_stab1} we
observe that,
\begin{multline}\label{time_diff_bound}
\int_0^T (\hat \nu \partial_x
u_h, \partial_x \partial_t u_h)_I ~\mbox{d}t \\= \int_0^T (\max(0,\hat \nu-\nu) \partial_x
u_h, \partial_x \partial_t u_h)_I ~\mbox{d}t + \int_0^T (\nu \partial_x
u_h, \partial_x \partial_t u_h)_I ~\mbox{d}t\\
\leq TV(u_h(\cdot,0)) \int_0^T \|\hat \nu \partial_x \partial_t
u_h\|_{L^\infty(I)} ~\mbox{d}t
+\frac12 \|\nu \partial_x u_h(\cdot,T)\|^2  - \frac12 \|\nu \partial_x
u_h(\cdot,0)\|^2\\
\lesssim U_0 \| \partial_x
u_h(\cdot,0)\| h^{-\frac12} T^{\frac12} \|\partial_t u_h\|_{L^2(Q)} \\+\frac12 \|\nu^\frac12 \partial_x u_h(\cdot,T)\|^2  - \frac12 \|\nu^\frac12 \partial_x
u_h(\cdot,0)\|^2.
\end{multline}
Hence by applying \eqref{time_transp_bound} and \eqref{time_diff_bound} in
the right hand side of \eqref{time_stab1} and norm equivalence in the
left hand side of \eqref{time_stab1} we obtain the bound
\begin{multline*}
\|\partial_t u_h\|^2_Q\leq C_q \int_0^T \|\partial_t u_h\|_h^2
~\mbox{d}t \\
\leq C_q^2 2 U_0^2  \| \partial_x
u_h(\cdot,0)\|^2 T h^{-1}+ \frac{1}{2}\|\partial_t
u_h\|^2_{L^2(Q)}+ C_q^2\frac{\nu }{2} \| \partial_x
u_h(\cdot,0)\|^2.
\end{multline*}
The conclusion is immediate.
\end{proof}

\section{The linearized dual adjoint}
We introduce the linearized adjoint problem
\begin{equation}\label{adjoint_pert_equation}
\begin{array}{rcl}
-\partial_t \varphi - a(u,u_h) \partial_x\varphi -
\nu \partial_{xx} \varphi &=& 0 \mbox{
  in } Q, \\[3mm]
\varphi(0,t) &=& \varphi(1,t) \mbox{
  for } t \in (0,T],\\[3mm]
\\[3mm]
\partial_x \varphi(0,t) &=& \partial_x \varphi(1,t) \mbox{
  for } t \in (0,T],\\[3mm]
\varphi(x,T) &=&  \psi(x) \mbox{
  for } x \in I,
\end{array}
\end{equation}
where $a(u,u_h):= (u+u_h)/2$.
The rationale for the dual adjoint is the following derivation of a
perturbation equation for the functional of the error
$|(e(T),\psi)_I|$, where $e(T):= u(T)-u_h(T)$.
\begin{multline}\label{error_rep}
|(e(T),\psi)_I| = |(e(T),\psi)_I+\int_0^T (e,-\partial_t \varphi - a(u,u_h) \partial_x\varphi -
\nu \partial_{xx} \varphi)_I ~\mbox{d}t| \\
= |(e(0), \varphi(0))_I +  \int_0^T (\partial_t e+ \partial_x
(a(u,u_h)  e),\varphi)_I~\mbox{d}t + \int_0^T (\nu \partial_x e, \partial_x \varphi)_I~\mbox{d}t|\\
= |(e(0), \varphi(0))_I- \int_0^T (\partial_t u_h + u_h \partial_x u_h, \varphi)_I~\mbox{d}t - \int_0^T (\nu \partial_x
u_h, \partial_x \varphi)_I~\mbox{d}t|.
\end{multline}
This relation connects the error to the computational residual
weighted with the solution to the adjoint problem and can lead both to
a posteriori error estimates and to a priori error estimates, provided
we have sufficient information on the stability properties of the
numerical discretization methods and of the dual problem. The
a posteriori error estimate uses techniques similar to the now
classical dual weighted residual method, however in our case we can
estimate the dual weights analytically, accounting for perturbations,
both in the
discrete and the continuous solution.
Combining the a posteriori bounds with strong
stability properties of the numerical method, leads to a priori
upper bounds of the a posteriori quantities, showing that these
must converge and in consequence that the error goes to
zero. Before proceeding with this analysis we derive an a priori
estimate for the derivatives of the dual adjoint \eqref{adjoint_pert_equation}. 
\subsection{Wellposedness and stability}
Since $a(u,u_h) \in W^{1,\infty}(I)$ the problem \eqref{adjoint_pert_equation}
has a unique solution and one may show that it satisfies the maximum
principle
\[
\max_{(x,t) \in Q} |\varphi(x,t)| \leq \max_{x \in I} |\psi(x)|.
\]
The following stability estimate follows easily by standard energy methods
\begin{lemma}\label{adjoint_stab}
Let $\varphi$ be the solution to \eqref{adjoint_pert_equation} then
there holds
\begin{equation}\label{adjoint_stability}
\sup_{t \in (0,T)} \|\partial_x \varphi(\cdot,t)\|^2 +
 \nu  \|\partial_{xx} \varphi\|_Q^2 \lesssim \exp({D_0 T}) \| \partial_x \psi\|^2.
\end{equation}
\end{lemma}
\begin{proof}
Multiply the equation \eqref{adjoint_pert_equation} by
$- \partial_{xx} \varphi$ and integrate over $I \times (t,T)$
\begin{equation}\label{first_rel}
\| \partial_x \varphi(\cdot,t)\|^2 
+ 2 \int_t^T  \nu \|  \partial_{xx} \varphi\|^2~\mbox{d}t = \|\partial_x \psi\|^2 - 2 \int_t^T
(a(u,u_h) \partial_x \varphi, \partial_{xx} \varphi)~\mbox{d}t.
\end{equation}
 Note that, by an integration by parts and the maximum principles
 \eqref{umax} and \eqref{grad_max}
\begin{multline*}
2 \int_t^T (a(u,u_h) \partial_x \varphi,  \partial_{xx}
\varphi)_I ~\mbox{d}t =
- \int_t^T (\partial_x a(u,u_h) \partial_x
\varphi, \partial_x \varphi) ~\mbox{d}t \\ \geq -D_0  \int_t^T
\|\partial_x \varphi\|^2 ~\mbox{d}t.
\end{multline*}
The result in $L^\infty(0,T;L^2(I))$ follows from the Gronwall's lemma and taking the supremum
over $t \in (0,T)$ of the resulting expression
\[
\|\partial_x \varphi(\cdot,t)\|^2 \lesssim \exp(
  D_0 t ) \| \partial_x \psi\|^2.
\]
 The result for the second derivatives then follows
by using this expression to bound the right hand side of
\eqref{first_rel}
\begin{multline*}
2 \int_0^T \nu \|  \partial_{xx} \varphi\|^2~\mbox{d}t \leq
\|\partial_x \psi\|^2 +  \int_0^T (\partial_x a(u,u_h) \partial_x
\varphi, \partial_x \varphi) ~\mbox{d}t \\
\lesssim \| \partial_x \psi\|^2 \Bigl(1 +  D_0 \int_0^T  \exp(
  D_0 t ) ~\mbox{d}t\Bigr) = \exp(
  D_0 T) \|\partial_x \psi\|^2.
\end{multline*}
\end{proof}
\section{Error estimates for filtered
  quantities}
We will consider the differential filter defined in \eqref{filter_def},
where $\delta$ denotes a filter width to be specified. 
The norm associated to the differential filter is
given by
\[
\tnorm{\tilde u}_\delta := (\|\delta \partial_x \tilde
u\|^2+ \|\tilde u\|^2)^\frac12.
\]
We introduce the filtered error $\tilde e := \tilde u - \tilde u_h$,
where $\tilde u$ and $\tilde u_h$ denote the filtered exact and
approximate solutions respectively obtained by solving \eqref{filter_def} with
$u$ and $u_h$ as right hand side.
The analysis uses the stability properties of the adjoint perturbation
equation (Lemma \ref{adjoint_stab}) and the stability properties of
the discrete problem  (Lemma \ref{discrete_stab}) to
derive first a posteriori error bounds for the filtered quantities and
then a priori bounds by upper bounding the a posteriori residuals, by
a priori quantities.
\begin{theorem}\label{thm:mainresult}
Let $u$ be the solution of \eqref{burger}, $u_h$ be the solution of
\eqref{art_visc_FEM}. Then the following holds:
\begin{itemize}
\item A posteriori upper bound
\begin{multline}\label{aposteriori}
\hspace{-1.0cm}\tnorm{\tilde e(T)}_\delta \lesssim \exp(D_0 T)\left(\frac{h}{\delta^2}\right)^\frac12
\Bigl(h^\frac12 \|(u-u_h)(0)\|+
h^\frac12  \int_0^T \inf_{v_h \in V_h} \|v_h + u_h \partial_x
u_h\|  ~\mbox{d}t \\+ h^{\frac32} \int_0^T  \|\partial_x \partial_t
u_h\| ~\mbox{d}t +  \int_0^T  \|\max(0,\hat \nu-\nu)^\frac12 \partial_x 
u_h\| ~\mbox{d}t \\
+ 
   h \Bigl(\int_0^T  \nu\|[\partial_x u_h]\|_N^2  ~\mbox{d}t\Bigr)^{\frac12}
   \Bigr),
\end{multline}
where 
\[
\|[\partial_x u_h]\|_N:= \left( \sum_{i=0}^{N-1} (\partial_x u_h(x_i)\vert_{I_{i+1}}
- \partial_x u_h(x_i)\vert_{I_{i}})^2 \right)^{\frac12},
\]
with $I_N$ identified with $I_0$ by periodicity.
\item A priori upper bound\\ 
\begin{multline}\label{apriori_Rehigh}
\tnorm{\tilde e}_\delta \lesssim  \exp(D_0 T) \left(\frac{h}{\delta^2}\right)^\frac12
\Bigl(\Bigl(h^\frac12 + U_0^\frac12 \sqrt{T} \Bigr)\|u_0\| +
 (T U_0 + h^{\frac12}\nu^\frac12) \| \partial_x u_0\|
\Bigr).
\end{multline}
\end{itemize}
\end{theorem}
\begin{proof}
Let $\psi = \tilde e(T)$ in the definition \eqref{adjoint_pert_equation} of the dual adjoint problem. 
Using the design of the dual problem and the definition of $\tilde e$ we have
by the relation \eqref{error_rep}
\begin{multline*}
\tnorm{\tilde e(T)}^2_\delta =
(\delta \partial_x \tilde e(T),\partial_x \tilde e(T))_I + (\tilde
e(T),\tilde e(T))_I = (e(T),\tilde e(T))_I
\\=(e(0), \varphi(0))_I - \int_0^T (\partial_t u_h + u_h \partial_x
u_h, \varphi)_I ~\mbox{d}t - \int_0^T (\nu \partial_x
u_h, \partial_x \varphi)_I ~\mbox{d}t.
\end{multline*}
Taking $v_h = \pi_h \varphi$, with $\pi_h$ denoting the standard $L^2$-projection, in \eqref{art_visc_FEM} and adding to the
above expression yields
\begin{multline*}
\tnorm{\tilde e}^2_\delta =(e(0), \varphi(0))_I - \int_0^T (\partial_t u_h + u_h \partial_x
u_h, \varphi)_I ~\mbox{d}t - \int_0^T (\nu \partial_x
u_h, \partial_x \varphi)_I ~\mbox{d}t\\
+ \int_0^T (\partial_t u_h,\pi_h \varphi) _h ~\mbox{d}t +\int_0^T ( u_h \partial_x
u_h, \pi_h \varphi)_I ~\mbox{d}t + \int_0^T (\hat \nu \partial_x
u_h, \partial_x \pi_h \varphi)_I ~\mbox{d}t\\
= \underbrace{ (e(0), \varphi(0))_I}_{T_0} \underbrace{ - \int_0^T (\partial_t u_h + u_h \partial_x
u_h, \varphi - \pi_h \varphi)_I ~\mbox{d}t}_{T_1} \\
\underbrace{ -\int_0^T ((\partial_t
u_h,\pi_h \varphi_h)_I - (\partial_t
u_h,\pi_h \varphi_h)_h) ~\mbox{d}t}_{T_2} \underbrace{- \int_0^T (\nu \partial_x
u_h, \partial_x (\varphi - \pi_h \varphi))_I ~\mbox{d}t}_{T_3}  \\ \underbrace{+ \int_0^T
(\max(0,\hat \nu(u_h) - \nu) \partial_x
u_h, \partial_x \pi_h \varphi)_I ~\mbox{d}t}_{T_4} .
\end{multline*}
Now consider the terms $T_0$ to $T_4$ term by term. First use the
orthogonality $(e_0,v_h)_I=0$ for all $v_h \in V_h$,
\[
T_0 = (e(0), \varphi-\pi_h \varphi)_I \lesssim \|e_0\|
{h} \sup_{t\in(0,T)}\|\partial_x \varphi(t)\|.
\]
Similarly for all $w_h \in
V_h$ there holds
\begin{multline*}
T_1 = - \int_0^T (w_h + u_h \partial_x
u_h, \varphi - \pi_h \varphi)_I ~\mbox{d}t \\ \leq \int_0^T \|w_h + u_h \partial_x
u_h\|  ~\mbox{d}t \sup_{t\in(0,T)}
\|(\varphi - \pi_h \varphi)(t)\| 
\end{multline*}
and hence
\[
T_1  \lesssim {h}\int_0^T \inf_{v_h \in V_h} \|v_h + u_h \partial_x
u_h\| ~\mbox{d}t \sup_{t\in(0,T)}\|\partial_x
\varphi(t)\|.
\]
Let $\mathcal{I}_h$ denote the standard Lagrange interpolant.
By the definition of the discrete $L^2$-inner product
$(\cdot,\cdot)_h$ we have 
\begin{multline*}
T_2 =-\int_0^T \int_I (\partial_t u_h \pi_h \varphi -
\mathcal{I}_h(\partial_t u_h \pi_h \varphi)) ~\mbox{d}x \mbox{d}t\\
\leq \int_0^T \int_I h^2 |\partial_x \partial_t u_h \partial_x \pi_h \varphi| ~\mbox{d}x\mbox{d}t 
\lesssim \int_0^T h^2 \|\partial_x \partial_t u_h\| \| \partial_x
\pi_h \varphi\| ~\mbox{d}t\\
\lesssim \int_0^T
h^2 \|\partial_x \partial_t u_h\| ~\mbox{d}t \sup_{t\in(0,T)}\| \partial_x \varphi(t)\| .
\end{multline*}
For $T_3$ we have after an integration by parts and using a trace
inequality followed by approximation
\begin{multline*}
T_3 = \int_0^T \sum_{i=0}^{N-1} \nu (\partial_x u_h(x_i)\vert_{I_{i+1}}
- \partial_x u_h(x_i)\vert_{I_{i}}) (\varphi(x_i) - \pi_h
\varphi(x_i))  ~\mbox{d}t
\\[3mm]
\lesssim \int_0^T \nu \|[\partial_x u_h]\|_N
(h^{-\frac12} \|\varphi- \pi_h \varphi \| + h^{\frac12} \|\partial_x (\varphi- \pi_h \varphi) \|)  ~\mbox{d}t\\[3mm]
\lesssim \left(\int_0^T  \nu\|[\partial_x u_h]\|_N^2  ~\mbox{d}t \right)^{\frac12}
h^{\frac32}
\|\nu^{\frac12} \partial_{xx} \varphi\|_Q
\end{multline*}
Finally the non-consistent artificial viscosity term is controlled
using the Cauchy-Schwarz inequality  and the $H^1$-stability of the $L^2$-projection $\|\partial_x \pi_h
\varphi\| \lesssim \|\partial_x 
\varphi\|$
\begin{multline*}
T_4 \leq \max_{i} \hat \nu \vert_{I_i} \int_0^T \|\max(0,\hat
\nu - \nu)^{\frac12} \partial_x u_h\|
~\mbox{d}t  \sup_{t \in (0,T)} \|\partial_x \varphi(t)\|\\[3mm]
\lesssim \left({U_0 h}\right)^\frac12 \int_0^T
  \|\max(0,\hat \nu-\nu)^\frac12\partial_x u_h\|
~\mbox{d}t \sup_{t \in (0,T)} \|\partial_x \varphi(t)\|.
\end{multline*}
Collecting the bounds for $T_0-T_4$ and using the stability estimate \eqref{adjoint_stability}
we have

\begin{multline}\label{apost_conclude}
\tnorm{\tilde e(T)}^2_\delta \lesssim \exp(D_0 T) \left(\frac{h}{\delta^2}\right)^\frac12
\Bigl(h^\frac12 \|e(0)\|+
h^\frac12  \int_0^T \inf_{v_h \in V_h} \|v_h + u_h \partial_x
u_h\|  ~\mbox{d}t\\+  U^\frac12_0 \int_0^T
  \|\max(0,\hat \nu-\nu)^{\frac12}\partial_x u_h\|
~\mbox{d}t+ h^{\frac32} \int_0^T  \|\partial_x \partial_t
u_h\| ~\mbox{d}t \\
+ 
   h \left(\int_0^T  \nu\|\jump{\partial_x u_h}\|_N^2  ~\mbox{d}t
   \right)^{\frac12} \Bigr) \tnorm{\tilde e(T)}_\delta
\end{multline}
from which \eqref{aposteriori} follows.

The a priori error estimate now follows by using discrete stability to
bound the residuals. Since we do not assume any regularity of the
exact solution we can not assume that any stronger bounds hold. Note
that by a well known discrete interpolation estimate \cite{Bu05} the convective
residual may be bounded, 
\[
\inf_{v_h \in V_h} \|h^{\frac12} (v_h - u_h \partial_x u_h)\| \lesssim
h \|\jump{u_h \partial_x u_h}\|_N. 
\]
Using that
\begin{multline*}
h [u_h \partial_x u_h]\vert_{x_i} \leq h |u_h(x_i)| \left(\frac{|\jump{\partial_x
  u_h}|}{2\{|\partial_x u_h|\}+\epsilon} \right)\vert_{x_i} \left(|\partial_x
u_h|\vert_{I_{i-1}} +|\partial_x
u_h|\vert_{I_{i}} + \epsilon\right)\\
\leq (\hat \nu(u_h) \partial_x u_h)\vert_{I_{i-1}}+ (\hat
\nu(u_h) \partial_x u_h)\vert_{I_{i}}+ h U_0 \epsilon
\end{multline*}
we may deduce
\[
\inf_{v_h \in V_h} \|h^{\frac12} (v_h - u_h \partial_x u_h)\|  \lesssim U_0^{\frac12} \|\hat
\nu^\frac12 \partial_x u_h\|+  h^{\frac12} U_0 \epsilon.
\]
In the linear case the inequality is trivial by taking $v_h=0$ and
using \eqref{umax}.

Then use the Cauchy-Schwarz inequality in time for the two terms of
the second line of \eqref{apost_conclude}, an inverse inequality for
the second term in the second line and
a trace inequality for the last term of \eqref{apost_conclude}. 
\begin{multline*}
\tnorm{\tilde e(T)}_\delta \lesssim  \exp( D_0 T) \left(\frac{h}{\delta^2}\right)^\frac12
\Bigl(h^\frac12 \|e(0)\|+
(h^\frac12 + U_0^\frac12 \sqrt{T}) \| \hat \nu^\frac12 \partial_x
u_h\|_Q \\+ h^{\frac12} \sqrt{T} \|\partial_t
u_h\|_Q + T^{\frac12} h^{\frac12} U_0 \epsilon\Bigr).
\end{multline*}
We conclude by applying the stability estimates and \eqref{energy_est} and
\eqref{timeder_stab}
leading to 
\begin{multline*}
\tnorm{\tilde e}_\delta \lesssim  \exp(D_0 T) \left(\frac{h}{\delta^2}\right)^\frac12
\Bigl((h^\frac12 + U_0^\frac12 \sqrt{T} )\|u_0\|\\
+  (T U_0 + h^{\frac12}\nu^\frac12) \|\partial_x u_0\| + T^{\frac12} h^{\frac12} U_0 \epsilon
\Bigr).
\end{multline*}
\end{proof}\\
Observe that the above estimate is independent both of the regularity of
the exact solution and of the flow regime.
\begin{remark}
Note that if instead the initial data $u_h(x,0)$ is chosen as the
nodal interpolant of $u_0$,  $\mathcal{I}_h u_0$, we may define
$U_0 = \max_{x \in I} |u_0(x)|$ and $D_0 = \max_{x \in
  I} \partial_x u_0(x)$. On the other hand we can no longer use 
$L^2$-orthogonality in the upper bound for $T_0$. It appears that
in that case we must use the maximum principle of the dual problem to obtain
\begin{multline*}
T_0 = (e(0),\varphi)_I \lesssim \|e(0)\|_{L^1(I)}
\|\varphi(0)\|_{L^\infty(I)} \\
\leq \|e(0)\|_{L^1(I)} \|\Psi\|_{L^\infty(I)} \lesssim \delta^{-1}
\|e(0)\|_{L^1(I)} \tnorm{\tilde e}_\delta\\
\lesssim \left(\frac{h^s}{\delta}\right) \|\partial^s_x
u_0\|_{L^1(I)} \tnorm{\tilde e}_\delta,\, s=1,2.
\end{multline*}
The global convergence order will be the same, but it appears that the error
contribution from the initial data will be larger and the factor $\|\partial^s_x u_0\|_{L^1(I)}$
must be added to the right hand side of \eqref{apriori_Rehigh}. Another downside to
this approach is that it only works in one space dimension, whereas
before only the energy stability estimates of Lemma \ref{discrete_stab} used one
dimensional inverse inequalities.
\end{remark}
\begin{remark}
We have kept the dependence on the regularization parameter $\epsilon$
in the above proof. This shows the effect of regularization on the
computational error under the assumption that \eqref{TVbound} still
holds under regularization.
\end{remark}
\begin{remark}
Since all these estimates are independent of $\nu$ they are
also valid for the purely hyperbolic case, with $u$ the entropy solution.
\end{remark}
\section{$L^p$-error estimates using interpolation}
Using the estimate of the filtered error of Theorem \ref{thm:mainresult} together
with the $TV$-a priori bound of the discrete solution of Corollary
\ref{TVbound} we may now use an interpolation argument to prove an estimate
of the error in $L^p$-norm. Below we omit the dependence in time in
all arguments for clarity of exposition.
\begin{theorem}
Let $u$ be the solution of \eqref{burger} and $u_h$ the solution of \eqref{art_visc_FEM}
for which the conclusion of Theorem \ref{thm:mainresult} holds. Then
for $t>0$,
\[
\|u  - u_h\|_{L^p(I)} \lesssim h^{\frac{1}{3p}}, \quad 1\leq p < \infty,
\]
where the hidden constant depends only on the constants of Theorem
\ref{thm:mainresult} and Corollary \ref{TVbound}.
\end{theorem}
\begin{proof}
Let $\tilde e$ be the filtered error obtained taking $\delta = 1$ in \eqref{filter_def}.
By definition and using the triangle inequality and Sobolev injection
there holds
\[
\|u-u_h\|_{L^p(I)} \leq \|\partial_{xx} \tilde
e\|_{L^p(I)} + \|\tilde e\|_{L^p(I)} \lesssim  \|\partial_{xx} \tilde
e\|_{L^p(I)} + \tnorm{\tilde e}_1.
\]
By the Galiardo-Nirenberg interpolation inequality there holds
\[
\|\partial_{xx} \tilde e\|_{L^p(I)} \lesssim 
\| \partial^3_x \tilde
e\|_{L^1(I)}^{1-\frac{2}{3p}} \|\partial_x \tilde e\|_{L^2(I)}^{\frac{2}{3p}}.
\]
Then by adding and subtracting $\partial_x \tilde e$ in
the first factor we have after a triangle inequality, using the definition of $\tilde e$
\[
\|\partial^3_x \tilde e\|_{L^1(I)}^{1-\frac{2}{3p}} \|\partial_x \tilde e\|_{L^2(I)}^{\frac{2}{3p}} \leq
(\|\partial_x (u - u_h)\|_{L^1(I)}+ \|\partial_x
\tilde e\|_{L^2(I)})^{1-\frac{2}{3p}}\|\partial_x \tilde e\|_{L^2(I)}^{\frac{2}{3p}}.
\]
We conclude using the bound $TV(u) \lesssim TV(u_0)$, Corollary
\ref{TVbound} and Theorem \ref{thm:mainresult} that
\[
\|u - u_h\|_{L^p(I)} \lesssim h^{\frac{1}{3p}} + h^{\frac12}.
\]
\end{proof}
\section{Comparison with the theory of Nessyahu-Tadmor}
Comparing with the estimates obtained in \cite{NT92} 
\begin{equation}\label{NessTadconv}
\|u - u_h\|_{L^p(I)} \leq h^{\frac{1}{2p}}
\end{equation}
we see that we are suboptimal by an order of $h^{\frac{1}{6p}}$. This
loss of convergence is due to the fact that we need to control the
$L^2$-norm of the convective residual using the stabilization
resulting in the classical loss of $h^{\frac12}$. If the
residual had been in $L^1$ it would have been a priori bounded by
Corollary \ref{TVbound} and we could recover the convergence
\eqref{NessTadconv}. This program is indeed possible to carry out as we
will show in this section. We follow the abstract framework proposed
in \cite[Section 4.2]{Tad98}. First we recall the $Lip$-norm and
the associated dual semi-norm,
\[
\|u\|_{Lip} := ess\,sup_{x \ne y} \left|\frac{u(x) - u(y)}{x-y}\right|,
\]
\[
\|u\|_{Lip'}:= \sup_{v \in Lip, \|v\|_{Lip} = 1} (u - \bar u,v), \quad\bar
u := \int_I u ~\mbox{d}x.
\]
We also need to measure functions one-sided Lipschitz continuity,
\[
\|u\|_{Lip+} := ess\,sup_{x \ne y} \left(\frac{u(x) - u(y)}{x-y}\right)_+.
\]
We can then prove the result,
\begin{proposition}
Let $u$ be the entropy solution of \eqref{burger} with $\nu = 0$ and
$u_h$ the solution of \eqref{art_visc_FEM} with $\nu=0$. Then there
holds
\[
\|u_h - u\|_{Lip'} \leq C(u_0,T) h
\]
and
\[
\|u - u_h\|_{L^p(I)} \leq  C(u_0,T) h^{\frac{1}{2p}}
\]
\end{proposition}
\begin{proof}
This follows from Theorem 4.1 and Corollary 4.1 of \cite{Tad98} once
we have verified that the scheme satisfies three properties:
\begin{enumerate}
\item The solutions $u_h$ are conservative: this follows immediately by
  testing with $v_h = 1$ in \eqref{art_visc_FEM} and recalling that
  mass-lumping is conservative.
\item The solutions $u_h$ are $Lip'$-consistent: we must verify that
\begin{equation}\label{initial_data_lip}
\|u_h(\cdot,0) - u(\cdot,0)\|_{Lip'} \leq C h 
\end{equation}
and
\begin{equation}\label{residual_lip}
\|\partial_t u_h +\frac12 \partial_x u_h^2\|_{Lip'(x,[0,T]]} \leq C h.
\end{equation}
\item The solution $u_h$ and the intial data are $Lip^+$ stable and
$Lip^+$ bounded respectively: the first part is a consequence of Lemma
\ref{DMP} and the second follows by our choice of a smooth initial
data. If the initial data is in the finite element space it is enough
to choose it $Lip^+$ bounded.
\end{enumerate}
It only remains to show that \eqref{initial_data_lip} and
\eqref{residual_lip} are satisfied. The first follows since by the
orthogonality of the $L^2$-projection
\[
(u_h(\cdot,0) - u(\cdot,0),(\phi - \bar \phi) - \mathcal{I}_h (\phi - \bar
\phi)) \leq h \|u_h(\cdot,0) - u(\cdot,0)\|_{L^1(I)} \|\phi\|_{Lip}.
\]
To prove \eqref{residual_lip} we need to prove that the time
derivative $\partial_t u_h$ is bounded in $L^1(I)$. This
bound is a consequence of the formulation \eqref{art_visc_FEM} and
Corollary \ref{TVbound}
\begin{multline}\label{dt_L1}
\|\partial_t u_h\|_{L^1(I)} \leq h \sum_{i} |\partial_t u_h(x_i)| \leq
C (\|\frac12 \partial_x u_h^2\|_{L^1(I)} + \|h^{-1} \hat
\nu \partial_x u_h\|_{L^1(I)})\\ \leq C U_0 \|\partial_x
u_h\|_{L^1(I)} \leq C U_0 TV(u_0).
\end{multline}
To prove \eqref{residual_lip} we note that by Galerkin orthogonality
there holds, for all $t$,
\begin{multline*}
(\partial_t u_h +\frac12 \partial_x u_h^2, \phi) = (\partial_t u_h
+\frac12 \partial_x u_h^2, \phi - \mathcal{I}_h \phi)_I + \int_I
(\partial_t u_h\mathcal{I}_h \phi - (\mathcal{I}_h (\partial_t u_h
\mathcal{I}_h \phi) ) ~\mbox{d}x\\
- (\hat \nu \partial_x u_h, \partial_x \mathcal{I}_h \phi)_I \\
\leq
h (\|\partial_t u_h\|_{L^1(I)} + U_0 \|\partial_x u_h\|_{L^1(I)}  + h
\|\partial_x \partial_t u_h\|_{L^1(I)} + h^{-1} \|\hat \nu \partial_x
u_h\|_{L^1(I)}) \|\phi\|_{Lip}.
\end{multline*}
Note that $h
\|\partial_x \partial_t u_h\|_{L^1(I)} + h^{-1} \|\hat \nu \partial_x
u_h\|_{L^1(I)}\lesssim\|\partial_t u_h\|_{L^1(I)} + U_0 \|\partial_x
u_h\|_{L^1(I)}$ and the claim follows using the bounds of Corollary
\ref{TVbound} and \eqref{dt_L1}.
\end{proof}\\
By interpolation estimates on the gradient may be obtained and by using post processing pointwise error estimates may be
obtained, we refer the interested reader to \cite{NT92,Tad98}.
\section{Numerical examples}
In this section we will study two numerical examples computed with
$\nu=0$. We first consider a problem with smooth initial data
\[
u_0 = \frac12 (\cos(\pi x)+1).
\]
We compute the solution at $T=0.5$, before shock formation and compute
the exact solution on a mesh with $6400$ mesh points using fixed
point iteration. The intial data and the final solutions are given in
Figure \ref{smooth_sol}. In Table \ref{smootheps0} errors in several
different norms are presented on four consequtive meshes.  Here
$\epsilon=0$ to machine precision. Experimental
convergence rates are given in parenthesis. In the following table (Table
\ref{smoothepsh}) we present the same results for $\epsilon=h$. The results are
similar, with the difference that the regularized method gives second
order convergence in all norms, whereas the one without regularization
exhibits a slight reduction in the order in the $L^2$-norm. This is
not surprising since formally the order of the method is $h^{\frac32}$
and the unregularized method adds nonconsistent first order viscosity
at local extrema.
\begin{figure}
\begin{center}
\includegraphics[width=0.25\paperwidth]{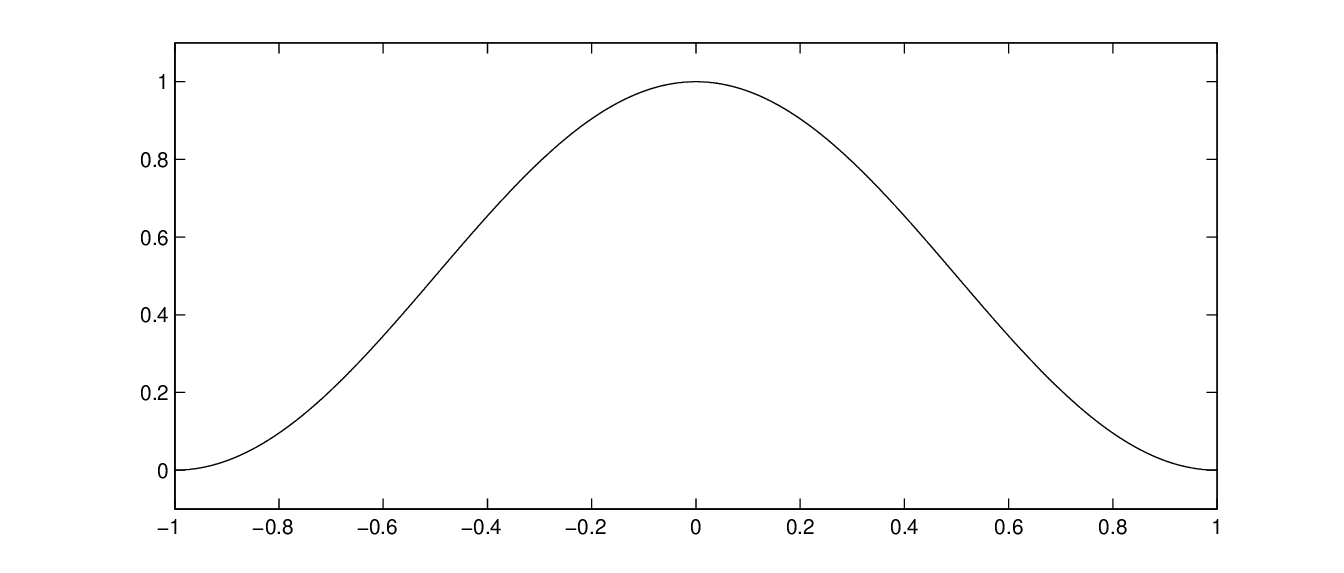}
\includegraphics[width=0.25\paperwidth]{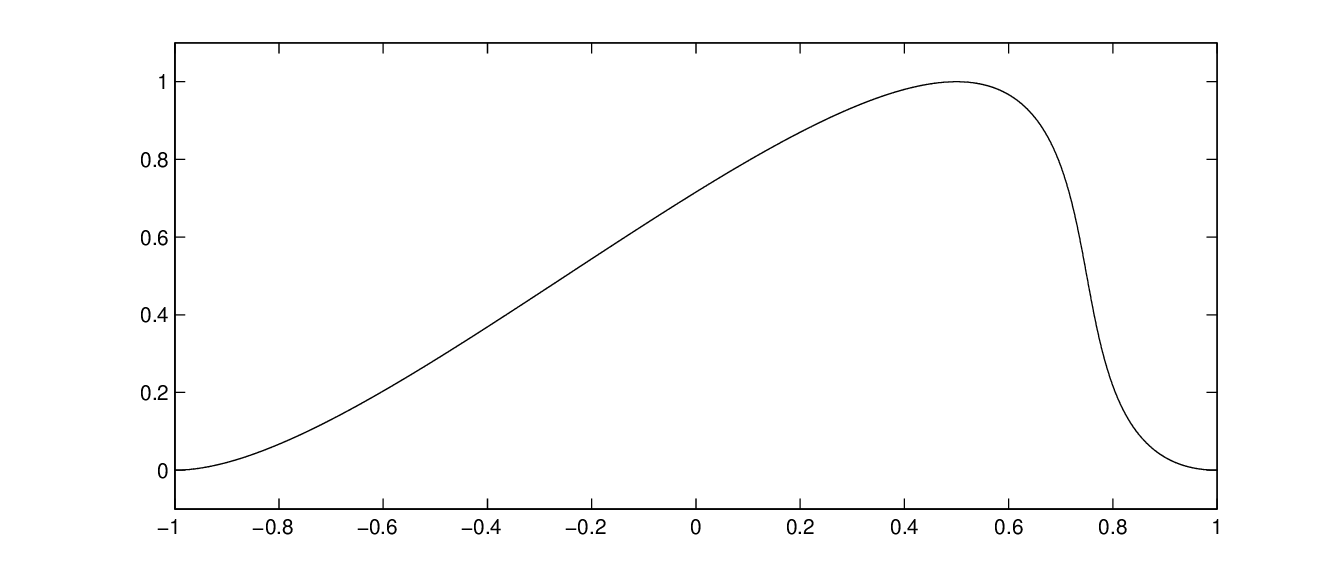}
\caption{Left smooth initial condition; right solution at $T=0.5$}
\label{smooth_sol}
\end{center}
\end{figure}
\begin{table}
\begin{center}
\begin{tabular}{|c|c|c|c|c|c|c|}
\hline
N & $\|u-u_h\|_{L^1(I)}$ & $\|u-u_h\|_{L^2(I)}$ & $\tnorm{\tilde
  e}_{1}$ $(\sim H^{-1}(I))$& $\tnorm{\tilde
  e}_{h}$ $(\sim L^2(I))$\\
\hline
100 & $2.5\cdot 10^{-3}$ & $3.6\cdot 10^{-3}$& $3.0\cdot 10^{-4}$ & $3.2\cdot 10^{-3}$\\
\hline
200 & $6.7\cdot 10^{-4}$ (1.9) & $1.0\cdot 10^{-3}$ (1.8) & $7.0\cdot
10^{-5}$ (2.1) &$9.5\cdot 10^{-4}$ (1.8)\\
\hline
400 & $1.8\cdot 10^{-4}$ (1.9) & $3.0\cdot 10^{-4}$ (1.7) & $1.7\cdot
10^{-5}$ (2.0) &$2.9\cdot 10^{-5}$ (1.7)\\
\hline
800 & $4.6\cdot 10^{-5}$ (2.0) & $8.9\cdot 10^{-5}$ (1.8) & $4.2\cdot
10^{-6}$ (2.0)& $8.7 \cdot 10^{-6}$ (1.7)\\
\hline
\end{tabular}
\end{center}
\caption{$\epsilon = O(10^{-16})$, smooth solution}\label{smootheps0}
\end{table}
\begin{table}
\begin{center}
\begin{tabular}{|c|c|c|c|c|c|}
\hline
N & $\|u-u_h\|_{L^1(I)}$ & $\|u-u_h\|_{L^2(I)}$ & $\tnorm{\tilde
  e}_{1}$ $(\sim H^{-1}(I))$& $\tnorm{\tilde
  e}_{h}$ $(\sim L^2(I))$\\
\hline
100 & $1.9\cdot 10^{-3}$ & $3.0\cdot 10^{-3}$& $2.3\cdot 10^{-4}$& $2.6\cdot 10^{-3}$\\
\hline
200 & $4.7\cdot 10^{-4}$ (2.0) & $7.7\cdot 10^{-4}$ (2.0) & $5.5\cdot
10^{-5}$ (2.1) &$7.1\cdot 10^{-4}$ (1.9)\\
\hline
400 & $1.2\cdot 10^{-4}$ (2.0) & $2.1\cdot 10^{-4}$ (1.9) & $1.3\cdot
10^{-5}$ (2.1)&$1.9\cdot 10^{-4}$ (1.8)\\
\hline
800 & $3.0\cdot 10^{-5}$ (2.0) & $5.5\cdot 10^{-5}$ (1.9) & $3.3\cdot
10^{-6}$ (2.0) &$5.3 \cdot 10^{-5}$ (1.8)\\
\hline
\end{tabular}
\end{center}
\caption{$\epsilon = h$, smooth solution}\label{smoothepsh}
\end{table}

Now we consider a problem with non-smooth solution. The initial data
and final time exact solution
is given in Figure \ref{rough_sol}. We compute the solution at
$T=0.5$ when the shock has formed. The exact solution is computed using the
method of characteristics on a mesh with 12800 elements. We present
tables with the same errors as in the previous case for the method
without (Table \ref{rougheps0}) and with (Table \ref{roughepsh})
regularization. For this case there is even less difference between
the two cases.
We observe
first order convergence for the $L^1$-error and the $H^{-1}$-norm
error
and $1/2$-order convergence in the $L^2$-norm. The computations
clearly show how the weaker norm behaves either as an $L^1$-norm for
$\delta=1$ or an $L^2$-norm for $\delta=h$. Intermediate values of
$\delta$
appears to interpolate between these two norms. This indicates that the
principle of our estimate, with the order depending on how $\delta$ is
chosen with respect to $h$ is correct. A superconvergence of
approximately half an order is observed for all the computations with
the nonsmooth solution, compared to what is predicted by theory.
 \begin{figure}
\begin{center}
\includegraphics[width=0.25\paperwidth]{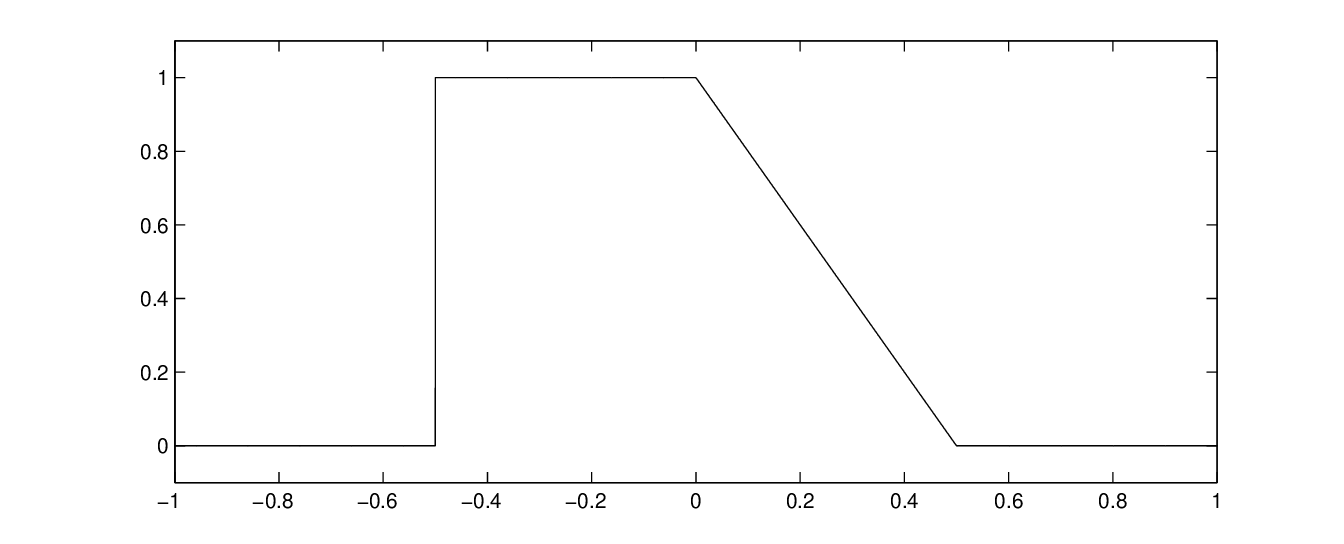}
\includegraphics[width=0.25\paperwidth]{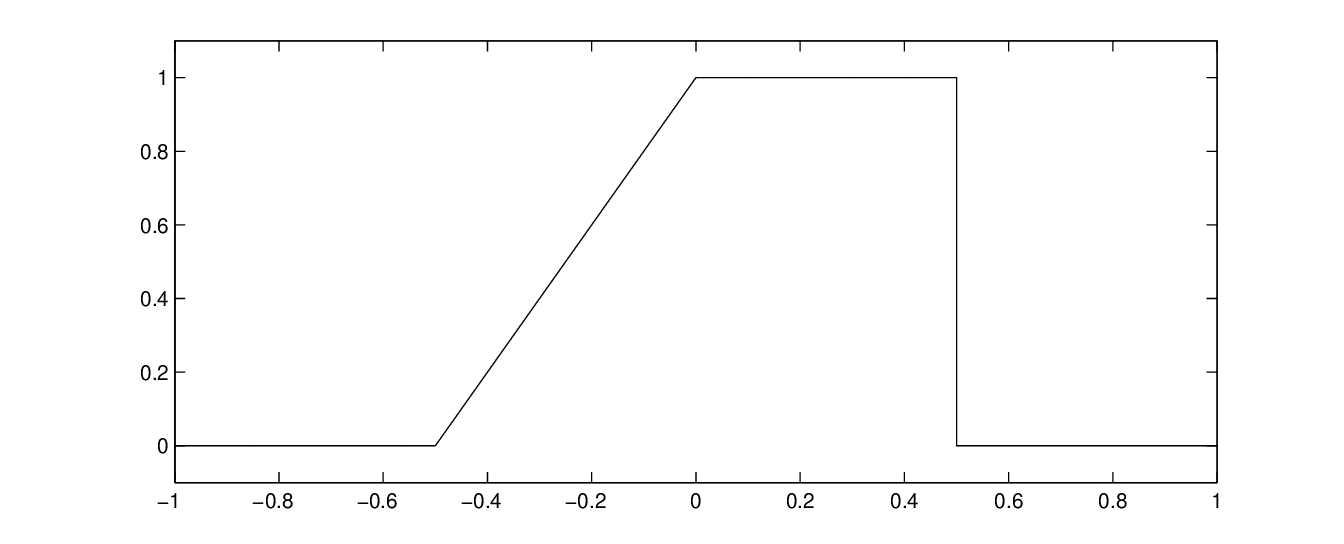}
\caption{Left nonsmooth initial condition; right solution at $T=0.5$}
\label{rough_sol}
\end{center}
\end{figure}
\begin{table}
\begin{center}
\begin{tabular}{|c|c|c|c|c|c|c|}
\hline
N & $\|u-u_h\|_{L^1(I)}$ & $\|u-u_h\|_{L^2(I)}$ &$\tnorm{\tilde
  e}_{1}$ $(\sim H^{-1}(I))$& $\tnorm{\tilde
  e}_{h}$ $(\sim L^2(I))$\\
\hline
100 & $0.036$ & $0.071$  & $6.4 \cdot 10^{-3}$ &$0.038$ \\
\hline
200 & $0.018$ (1.0) & $0.049$ (0.5) & $3.2 \cdot 10^{-3}$ (1.0) &$0.024$ (0.7) \\
\hline
400 & $9.4 \cdot 10^{-3}$ (0.9) & $0.034$ (0.5) & $1.6 \cdot 10^{-3}$
(1.0) & $0.016$ (0.6) \\
\hline
800 & $4.7 \cdot 10^{-3}$ (1.0) & $0.023$ (0.6) & $7.9 \cdot 10^{-4}$
(1.0) & $0.011$ (0.5)\\
\hline
\end{tabular}
\end{center}
\caption{$\epsilon = O(10^{-16})$, nonsmooth solution}\label{rougheps0}
\end{table}
\begin{table}
\begin{center}
\begin{tabular}{|c|c|c|c|c|c|c|}
\hline
N & $\|u-u_h\|_{L^1(I)}$ & $\|u-u_h\|_{L^2(I)}$ & $\tnorm{\tilde
  e}_{1}$ $(\sim H^{-1}(I))$& $\tnorm{\tilde
  e}_{h}$ $(\sim L^2(I))$\\
\hline
100 & $0.035$ & $0.070$  & $6.3 \cdot 10^{-3}$ & $0.037$\\
\hline
200 & $0.018$ (1.0) & $0.048$ (0.5) & $3.2 \cdot 10^{-3}$ (1.0) &$0.024$ (0.6)\\
\hline
400 & $9.1 \cdot 10^{-3}$ (1.0) & $0.033$ (0.5) & $1.6 \cdot 10^{-3}$ (1.0) & $0.016$ (0.6) \\
\hline
800 & $4.6 \cdot 10^{-3}$ (1.0) & $0.023$ (0.5) & $7.9 \cdot 10^{-4}$ (1.0) & $0.010$ (0.7)\\
\hline
\end{tabular}
\end{center}
\caption{$\epsilon = h$, nonsmooth solution}\label{roughepsh}
\end{table}

\bibliographystyle{plain}   

\end{document}